\documentclass[nobib]{tufte-handout} 
\usepackage{amsmath,stmaryrd,amssymb,amsthm,url,booktabs,hyperref,enumerate}
\usepackage{color}

\usepackage{stackrel}

\usepackage[style=alphabetic,%
            citestyle=alphabetic]{biblatex}
\makeatletter
\renewcommand{\@tufte@reset@par}{%
  \setlength{\RaggedRightParindent}{0pc}
  \setlength{\JustifyingParindent}{0pc}
  \setlength{\parindent}{0pc}
  \setlength{\parskip}{0pt}%
}
\@tufte@reset@par

\makeatother

\ifpdf 
	\usepackage[all,pdf,cmtip]{xy} 
\else
	\input xy 
	\xyoption{all} 
	\xyoption{2cell} 
	\xyoption{v2}
\fi

\def\into {\hookrightarrow}

\def\cH {\mathcal{H}}

\def\RR{\mathbb{R}} 
\def\NN{\mathbb{N}}
\def\ZZ{\mathbb{Z}}
\def\QQ{\mathbb{Q}}

\def\tors{\mathrm{tors}}

\DeclareMathOperator*{\colim}{colim}

\DeclareMathOperator{\Hom}{Hom} 
\DeclareMathOperator{\id}{id} 
 
\DeclareMathOperator{\coker}{coker}
 
\DeclareMathOperator{\Aut}{Aut} 

\DeclareMathOperator{\im}{im}
\DeclareMathOperator{\pr}{pr}

\DeclareMathOperator{\Spin}{Spin} 

\DeclareMathOperator{\Ext}{Ext}

\newtheorem{thm}{Theorem}[section]
\newtheorem{prop}[thm]{Proposition}
\newtheorem{cor}[thm]{Corollary}
\newtheorem{lemma}[thm]{Lemma}
\newtheorem{defn}[thm]{Definition}
\newtheorem{rem}[thm]{Remark}
\numberwithin{equation}{section}
\newtheorem{example}[thm]{Example}

\title{Many finite-dimensional lifting bundle gerbes are torsion.\thanks{This document is released under a Creative Commons Zero license \href{http://creativecommons.org/publicdomain/zero/1.0/}{\texttt{creativecommons.org/publicdomain/zero/1.0/}}}
}
\author{David Michael Roberts\thanks{\href{https://orcid.org/0000-0002-3478-0522}{orcid.org/0000-0002-3478-0522}
}}
\date{April 26, 2021} 

\begin{document}

\maketitle

\begin{abstract}
Many bundle gerbes constructed in practice are either infinite-dimensional, or finite-dimensional but built using submersions that are far from being fibre bundles.
Murray and Stevenson proved that gerbes on simply-connected manifolds, built from finite-dimensional fibre bundles with connected fibres, always have a torsion $DD$-class.
In this note I prove an analogous result for a wide class of gerbes built from principal bundles, relaxing the requirements on the fundamental group of the base and the connected components of the fibre, allowing both to be nontrivial.
This has consequences for possible models for basic gerbes, the classification of crossed modules of finite-dimensional Lie groups, the coefficient Lie-2-algebras for higher gauge theory on principal 2-bundles, and finite-dimensional twists of topological $K$-theory.
\end{abstract}

\section{Introduction}

A bundle gerbe \cite{Mur} is a geometric object that sits over a given space or manifold $X$ classified by elements of $H^3(X,\ZZ)$, in the same way that (complex) line bundles on $X$ are classified by elements of $H^2(X,\ZZ)$.
And, just as line bundles on manifolds have connections giving rise to \emph{curvature}, a 2-form, giving a class in $H^3_{dR}(X)$, bundle gerbes have a notion of geometric `connection' data, with curvature a 3-form and hence a class in $H^3_{dR}(X)$.
Since de~Rham cohomology sees only the non-torsion part of integral cohomology, bundle gerbes that are classified by torsion classes in $H^3$ are thus trickier in one sense to `see' geometrically.
The problem is compounded by the fact that bundle gerbes with the same 3-class may look wildly different, as the correct notion of equivalence is much coarser than isomorphism.
Thus different constructions that lead to the same class are still of interest, due to the flexibility this introduces.

One wide class of bundle gerbes---so-called \emph{lifting bundle gerbes}---arise from the following data. Given a Lie group $G$, a principal $G$-bundle $P\to X$, and a central extension 
\begin{equation}\label{eq:centralext}
U(1)\to \widehat{G} \to G
\end{equation} 
of Lie groups, there is a bundle gerbe on $X$ that is precisely the obstruction to the extension of a $\widehat{G}$-bundle lifting $P$. 
One can see the cohomology class corresponding to the bundle gerbe as analogous to the class $w_2\in H^2(M,\ZZ/2)$ obstructing the lifting of the frame bundle $F(M)$ to be a spin bundle.

For a non-trivial example, one can consider the central extension $U(1) \to U(n) \to PU(n)$, and $P$ a principal $PU(n)$-bundle. 
The lifting bundle gerbe associated to such a $PU(n)$-bundle has torsion class in $H^3(X,\ZZ)$.
Conversely, by a result of Serre, and published by Grothendieck \cite[I, Th\'eor\`eme 1.6]{Groth_Br_123}, every torsion class in $H^3(X,\ZZ)$ is associated to at least one lifting bundle gerbe of this form.\footnote{It is in fact nontrivial to find which ranks $n$ this is possible for, given a specific torsion element}

In the other direction, one can consider the extension $U(1) \to U(\cH) \to PU(\cH)$ of infinite-dimensional groups, where $\cH\simeq L^2[0,1]$, and lifting bundle gerbes of principal $PU(\cH)$-bundles. 
These bundle gerbes are infinite-dimensional, and \emph{every} class in $H^3(X,\ZZ)$ can be realised by \emph{some} lifting bundle gerbe of this form.
There are also constructions of bundle gerbes on compact, simply-connected, simple Lie groups $G$ that use lifting bundle gerbes for the infinite-dimensional Kac--Moody central extensions $U(1) \to \widehat{\Omega_kG} \to \Omega G$, and these are also non-torsion gerbes.

Between these examples, then, one might wonder when an a priori \emph{given} finite-dimensional lifting bundle gerbe on $X$ has a torsion class in $H^3(X,\ZZ)$.
That is, given a finite-dimensional central extension \eqref{eq:centralext} and a principal $G$-bundle $P\to X$, is the lifting bundle gerbe classified by a torsion class?
The answer is: not always, as there are explicit and easy examples where it is non-torsion, for instance over $X=S^2\times S^1$ (cf Example~\ref{ex:cup_product} below).
However, under a mild condition on $X$, it \emph{is} true.

Murray's original paper had a claim about a sufficient condition for a finite-dimensional bundle gerbe to be torsion, though the proof had a subtle error.
Murray and Stevenson later \cite{Murray_Stevenson_11} gave a correct proof of a slightly stronger result, demanding (i) simple-connectivity of the base manifold $X$, and (ii) that the fibres of a certain submersion $Y\to X$ (part of the bundle gerbe data) are connected.
This result was not specifically about lifting bundle gerbes, but it suffices to prove that lifting gerbes for $G$-bundles on simply-connected spaces $X$, where $G$ is a  \emph{connected} Lie group, are torsion.

However, the proof in \emph{loc.\ cit.}\ is more general, and doesn't use anything specific about lifting bundle gerbes, which are somewhat more rigid than the general case.
In this note I shall prove the following result:

\begin{thm}\label{thm:main_thm_easy}
Given any connected manifold $X$ with finite fundamental group, any central extension \eqref{eq:centralext} of finite-dimensional Lie groups, and any principal $G$-bundle $P\to X$, the corresponding lifting bundle gerbe is torsion.
\end{thm} 

In fact, the proof suffices to give a stronger result, which applies to topological spaces, assuming a group theoretical fact about homomorphisms $\pi_1(X)\to \pi_0(G)$.
We shall give this result below as Theorem~\ref{thm:main_thm_sharp}, and some applications in the final section.

\paragraph{Acknowledgements} 
I thank Konrad Waldorf for asking me whether all finite-dimensional lifting bundle gerbes are torsion, inspiring this note, and Maxime Ramzi for patient, extensive discussions. David Baraglia and MathOverflow users\footnote{\url{https://mathoverflow.net/q/389703}} also provided helpful information.

This work is supported by the Australian Research Council's \emph{Discovery Projects} funding scheme (grant number DP180100383), funded by the Australian Government.

\section{Preliminaries}

We first recall the definition of a bundle gerbe.

\begin{defn}[\cite{Mur}]
A bundle gerbe on a manifold $M$ consists of the following data:
\begin{itemize}
\item A surjective submersion $Y\to M$.
\item A principal $U(1)$-bundle $E\to Y^{[2]}:=Y\times_M Y $.
\item An isomorphism $\mu\colon \pr_{12}^*E\otimes_{U(1)} \pr_{23}^*E \to \pr_{13}^*E$ of $U(1)$-bundles\footnote{The tensor product here is defined, for arbitrary $U(1)$-bundles $E,F$ on $X$, to be $(E\times_X F)/U(1)$, where $U(1)$ acts via the anti-diagonal action.}
 on $Y^{[3]}:=Y\times_MY\times_MY$, called the bundle gerbe multiplication.
\item This isomorphism needs to satisfy an associativity condition, namely that the diagram\footnote{I have suppressed some canonical isomorphisms for clarity; all tensor products are over $U(1)$}
\[
  \xymatrix{
    E_{12}\otimes E_{23}\otimes E_{34} \ar[r]^{\id\times \mu} \ar[d]_{\mu\times \id} & E_{12}\otimes E_{24} \ar[d]^{\mu}\\
    E_{13}\otimes E_{34} \ar[r]_{\mu} & E_{14}
  }
\]
of isomorphisms of $U(1)$-bundles on $Y^{[4]}$ needs to commute. Here $E_{ij} = \pr_{ij}^*E$, with $\pr_{ij}\colon Y^{[4]}\to Y^{[2]}$
\end{itemize}
A bundle gerbe will be denoted by $(E,Y)$, the other data implicit.
A bundle gerbe on a \emph{topological space} is defined the same way, except in that case we require instead that $Y\to M$ admits local sections.
\end{defn}

Here is the only type of example we will need. Fix a locally trivial central extension $U(1)\to \widehat{G} \to G$ of Lie groups.\footnote{That is, ignoring the group structures on $\widehat{G}$ and $G$, it is a principal $U(1)$-bundle} Recall that the multiplication map of $\widehat{G}$ induces an isomorphism $\pr_1^*\widehat{G}\otimes_{U(1)}\pr_2^*\widehat{G} \simeq m^*\widehat{G}$ over $G\times G$ where $m\colon G\times G\to G$ is the multiplication map.

\begin{example}[\cite{Mur}]
Let $P\to X$ be a principal $G$-bundle. 
The lifting bundle gerbe associated to this bundle (and the fixed central extension) is given by the data:
\begin{itemize}
\item The submersion is $P\to X$. Recall that the action map induces an isomorphism $P\times G\simeq P\times_X P$ \emph{over $X$} we will use this silently from now on.
\item The $U(1)$-bundle is $P\times \widehat{G} \to P\times G$.
\item Using the isomorphism $P\times G\times G\to P^{[3]}$, $(p,g,h)\mapsto (p,pg,pgh)$, the multiplication is given by the composite
\[
\pr_{12}^*(P\times \widehat{G})\otimes_{U(1)} \pr_{23}^*(P\times \widehat{G}) \simeq P \times \pr_1^*\widehat{G}\otimes_{U(1)}\pr_2^*\widehat{G} \simeq P\times m^*\widehat{G}\simeq \pr_{13}^*(P\times \widehat{G})
\]
\item The associativity condition follows from associativity in the group $\widehat{G}$.
\end{itemize}
\end{example}

Given a bundle gerbe on a space $X$, and a map $f\colon X'\to X$, there is a bundle gerbe on $X'$ given by the data $f^*Y\to X'$ and $(p^{[2]})^*E$, where $p\colon f^*Y\to Y$ is the projection.
Moreover, the maps 
\[
  (p^{[2]})^*E \to E,\quad (f^*Y)^{[2]} \to Y^{[2]},\quad p\colon f^*Y\to Y
\]
are compatible with all the bundle gerbe structure.
This is the pullback of $(E,Y)$ by $f$, denoted $f^*(E,Y)$.
If $(P\times \widehat{G},P)$ is a lifting bundle gerbe, then $f^*(P\times \widehat{G},P) \simeq (f^*P\times \widehat{G},f^*P)$.

More generally, given bundle gerbes $(F,Z)$ and $(E,Y)$ on $X$ one can take a map $g\colon Z\to Y$ commuting with the projections to $X$, and a map $k\colon F\to E$ of $U(1)$-bundles covering $g^{[2]}$ as in the diagram
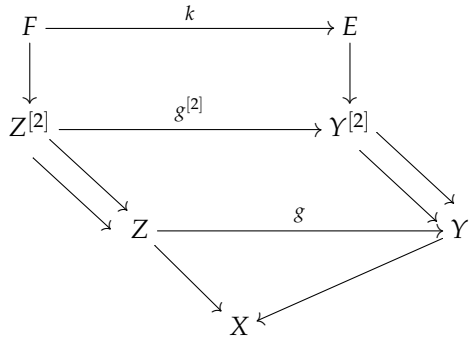
\[
  \xymatrix{
    F \ar[d] \ar[rrr]^k &&& E \ar[d]\\
    Z^{[2]} \ar@<+0.5ex>[dr] \ar@<-1.5ex>[dr] \ar[rrr]^{g^{[2]}} &&& Y^{[2]} \ar@<+1.3ex>[dr] \ar@<-0.7ex>[dr]\\
    & Z \ar[rrr]^g \ar[dr] &&& Y \ar[dll]\\
    && X  && 
  }
\]
such that the bundle gerbe multiplications are respected. 
The data of $g$ and $k$ is then a morphism of bundle gerbes, denoted $(k,g)\colon (F,Z)\to (E,Y)$
Note that in this case, we have $F\simeq (g^{[2]})^*E$, as $k$ is a map of principal bundles.
An \emph{isomorphism} $(E,Y)\simeq (F,Z)$ of bundle gerbes on $X$ will be meant in the strictest possible sense, namely where $g$ and $k$ are isomorphisms.

\begin{defn}
Given a bundle gerbe $(E,Y)$, a \emph{stable trivialisation} consists of a principal $U(1)$-bundle $T\to Y$ and an isomorphism $\pr_1^*T^*\times \pr_2T \simeq E$ of bundles on $Y^{[2]}$, making an isomorphism of bundle gerbes along with $\id_Y$.
\end{defn}

\begin{example}
Given a lifting bundle gerbe associated to $P\to X$ and $U(1) \to \widehat{G}\to G$, a stable trivialisation is equivalent data to a principal $\widehat{G}$-bundle $\widehat{P}\to X$ lifting $P$: the quotient map $\widehat{P}\to P$ is the required principal $U(1)$-bundle, and vice versa.
\end{example}

There is a general notion of tensor product of a pair of bundle gerbes, analogous to the tensor product of $U(1)$-bundles, but we will not need it here. 
However, the notion of a power of a single given bundle gerbe is easier to describe, namely if $(E,Y)$ is a bundle gerbe, then $(E,Y)^{\otimes n} := (E^{\otimes n},Y)$ is also a bundle gerbe, using the tensor powers of the $U(1)$-bundle $E$.

\begin{example}\label{ex:powers_of_lifting_gerbes}
Let $P\to X$ and $U(1)\to \widehat{G}\to G$ be the data necessary to build a lifting bundle gerbe.
Then $(P\times \widehat{G},P)^{\otimes n} \simeq (P\times \widehat{G}^{\otimes n},P)$, where $U(1)\to \widehat{G}^{\otimes n} \to G$ is the $n$-fold central product of $\widehat{G}$ with itself, which is the same as the $n$-fold tensor power of the underlying $U(1)$-bundle over $G$.
\end{example}

Associated to each bundle gerbe $(E,Y)$ on $X$ there is a class $DD(E,Y) \in H^3(X,\ZZ)$---the Dixmier--Douady, or $DD$-class---that satisfies the following properties (see \cite{Mur}, and \cite{Mur_Stev_stable_iso} for the last item).
\begin{lemma}\label{lemma:DD-facts}
\begin{enumerate}
\item Given $f\colon X'\to X$, then $f^*DD(E,Y) = DD(f^*(E,Y))$;
\item Given a morphism $(k,g)\colon (F,Z)\to (E,Y)$ of bundle gerbes on $X$, then $DD(F,Z)=DD(E,Y)$;
\item For all integers $n$, $DD((E,Y)^{\otimes n})=n\cdot DD(E,Y)$, where a negative tensor power involves the dual $U(1)$-bundle;
\item We have $DD(E,Y)=0$ if and only if the bundle gerbe $(E,Y)$ has a stable trivialisation;
\item\label{lemma:universal_lifting_gerbe} Given a central extension $U(1)\to \widehat{G}\to G$, classified by $\alpha\in H^3(BG,\ZZ)\simeq H^2(BG,\underline{U(1)})$, and a principal $G$-bundle classified by some map $\chi\colon X\to BG$, the lifting bundle gerbe has class in $H^3(X,\ZZ)$ corresponding to $\chi^*\alpha$.
\end{enumerate}
\end{lemma}

We shall say a bundle gerbe $(E,Y)$ is \emph{torsion} if $DD(E,Y)$ is a torsion element, and it is bundle gerbes of this form that are the main focus of this note.
The reader can rephrase all of the results of this paper without mentioning gerbes, if desired, using the following easy result.

\begin{cor}
Given a principal $G$-bundle $P\to X$ and a central extension $U(1)\to \widehat{G}\to G$, the associated lifting bundle gerbe $(P\times \widehat{G},P)$ is torsion if an only if there some $n>0$ such that the $G$-bundle $P$ lifts to a principal $\widehat{G}^{\otimes n}$-bundle.
\end{cor}

\section{Main results}

This first, relatively short, proof serves to illustrate the idea of the more complex proof of Theorem~\ref{thm:main_thm_sharp}.

\begin{proof}[Proof of Theorem~\ref{thm:main_thm_easy}]
Let us fix a lifting bundle gerbe $(E,P)$ associated to a principal $G$-bundle $P\to X$ and an extension $U(1)\to \widehat{G}\to G$. 
Let $G_0$ denote the connected component of the identity.
We can induce a principal $\pi_0(G)$-bundle on $X$ by defining $Q := P/G_0$, where we use the fact $\pi_0(G) = G/G_0$.
Let $\pi\colon \tilde{X}\to X$ denote the universal covering space of $X$, and recall that for any covering space on $X$, the pullback to $\tilde{X}$ is trivialisable.
Hence if we form the covering space $\pi^*Q \to \tilde{X}$, it is trivialisable.
From this it follows that the structure group of the $G$-bundle $\tilde{P}:=\pi^*P\to \tilde{X}$ reduces to $G_0$.
Thus we can find a subbundle $P' \subset \tilde{P}$ that on fibres looks like the inclusion $G_0 \into G$.

If we denote by $\widehat{G}_0\subset \widehat{G}$ the preimage of $G_0$, then we can form the lifting bundle gerbe on $\tilde{X}$ associated to $P'\to \tilde{X}$ and the central extension $U(1) \to \widehat{G}_0\to G_0$.
By the construction of a lifting bundle gerbe, we get that there is a morphism of bundle gerbes on $\tilde{X}$:
\[
  \xymatrix@C=2ex{
    P'\times \widehat{G}_0 \ar[d] \ar[rrr] &&& P\times \widehat{G} \ar[d]\\
    P'\times G_0 \ar@<+0.5ex>[dr] \ar@<-1.5ex>[dr] \ar[rrr] &&& P\times G \ar@<+0.5ex>[dr] \ar@<-1.5ex>[dr]\\
    & P' \ar[rrr] &&& P
  }
\]
Thus the DD-class $DD(P')$ of the lifting gerbe of $P'$ equal to $DD(\tilde{P})$.

However, the central extension $U(1) \to \widehat{G}_0\to G_0$ is classified by an element of $\alpha\in H^3(BG_0,\ZZ)$.
This cohomology group is pure torsion \cite[Lemme 26.1]{Borel_53}, where we use the fact that the cohomology of the classifying space of a connected Lie group is isomorphic to the cohomology of the classifying space of the maximal compact subgroup.
More precisely, the torsion subgroup of $H^*(BG_0,\ZZ)$ is exactly the kernel of the restriction map $H^*(BG_0,\ZZ)\to H^*(BT,\ZZ)$ \cite{Feshbach_81}, where $T\subseteq G_0$ is a maximal torus of $G_0$ (and $H^*(BT,\ZZ)$ is torsion-free).
Now $DD(P')$ is the image of $\alpha$ under $\chi^*\colon H^3(BG_0,\ZZ)\to H^3(\tilde{X},\ZZ)$, where $\chi\colon \tilde{X}\to BG_0$ is a classifying map.
Thus the lifting gerbe associated to $P'$ is torsion, hence so is the lifting gerbe (on $\tilde{X}$) associated to $\tilde{P}$.

Now we can apply the following lemma, as $\pi_1(X)$ is finite, and conclude that the lifting gerbe associated to $P$ is also torsion, since $DD(\tilde{P}) = \pi^*DD(P)$.
\end{proof}

\begin{lemma}
Given a $k$-sheeted covering space $\pi\colon Y\to X$, and a class $c\in H^n(X,\ZZ)$, if $\pi^*c\in H^n(Y,\ZZ)$ is torsion, then $c$ is torsion.
\end{lemma}

\begin{proof}
The composite $H^n(X,\ZZ)\otimes \QQ \xrightarrow{\pi^*} H^n(Y,\ZZ)\otimes \QQ \xrightarrow{\pi_*} H^n(X,\ZZ)\otimes \QQ$ is multiplication by $k$, which is invertible, so the pushforward map $\pi_*$ is a retraction of $\pi^*$, hence the latter is injective.

Then given a class $c\in H^n(X,\ZZ)$, if $\pi^*c$ is torsion, then $\pi^*c\otimes\QQ$ is zero, hence $c\otimes \QQ$ is zero, hence $c$ is a torsion class.
\end{proof}

\begin{rem}
The proof of Theorem 3.6 of \cite{Murray_Stevenson_11} actually tells us a tiny bit more than is claimed.
What it shows is that instead of demanding the base space $M$ is simply-connected, it is sufficient to ask that the composite $\pi_3(M)\to H_3(M)\to H_3(M)/\tors$ is surjective; this is because $\Hom(H_3(M)/\tors,\ZZ) \simeq\Hom(H_3(M),\ZZ)$, and it is the latter group that is used in the proof.

Further, the previous Lemma could be used to instead only ask that there is a finite-sheeted covering space $\tilde{M}\to M$ with $\pi_3(\tilde{M})\to H_3(\tilde{M})/\tors$ a surjective map. 
\end{rem}

Returning to the case of lifting gerbes, we can make a sharper claim, leaving the smooth category altogether.
We first need a preparatory technical lemma.

\begin{lemma}\label{lemma:torsion_extensions}
Let $K$ be a discrete group such that $H_3(K,\ZZ)$, $\Ext(H_2(K,\ZZ),\ZZ)$ and $H_1(K,A)$ are torsion, where $A$ is any $K$-module with underlying abelian group $A\simeq \ZZ^n$. 
If $G$ is a finite-dimensional Lie group with $\pi_0(G)\simeq K$, then $H^3(BG,\ZZ)$ is torsion.
\end{lemma}

\begin{proof}
For the rest of the proof, unspecified coefficients for cohomology should be taken as $\ZZ$. Given a Lie group as in the statement, there is a short exact sequence of Lie groups $G_0\to G\xrightarrow{\pi} K$, where $G_0$ is the connected component of the identity.
This gives a fibration $BG_0\to BG\xrightarrow{B\pi} BK$, where the fibre $BG_0$ is connected.
There is a (first quadrant) Serre spectral sequence for this fibration given by 
\[
  E_2^{s,t} = H^s(K,\mathcal{H}^t(BG_0)) \Longrightarrow H^{s+t}(BG),
\]
where $\mathcal{H}^t(BG_0)$ is the $K$-module arising from the local system on $BK$ with fibre over $x\in BK$ given by $H^t(B\pi^{-1}(x))\simeq H^t(BG_0)$.
Note in particular that $\mathcal{H}^0(BG_0)\simeq \ZZ$ carries the trivial $K$-action, since it can be identified with the constant $\ZZ$-valued functions on $BG_0$.
Recall that $H^k(BG_0)$ is finitely generated and is torsion for odd $k$ \cite{Borel_53}.

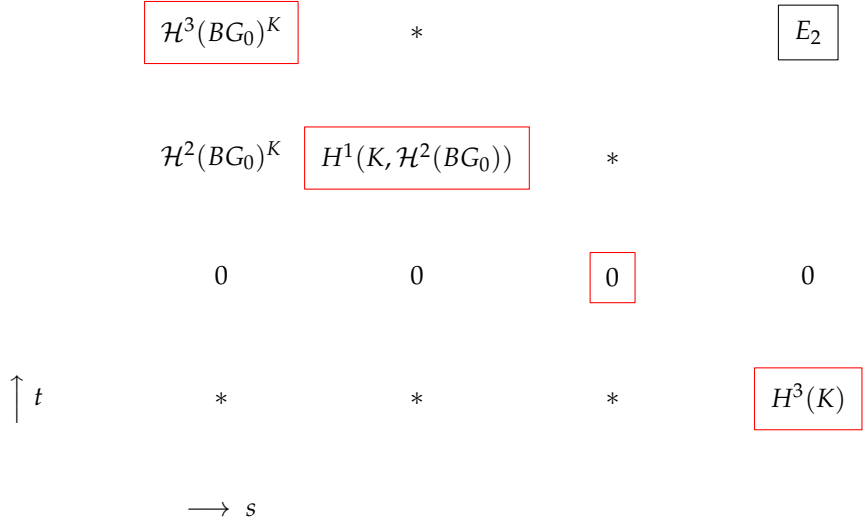
\begin{figure}
\[
  \xymatrix@!C=5em{
    & *++[F-:red]\hbox{$\cH^3(BG_0)^K$} & \ast &  & *++[F-:black]\hbox{$E_2$} \\
    & \cH^2(BG_0)^K
    & *++[F-:red]\hbox{$H^1(K,\mathcal{H}^2(BG_0))$} 
    & \ast & \\
    &0 & 0 & *++[F-:red]\hbox{$0$} & 0  \\
    \Big\uparrow\ t &\ast & \ast & \ast & *++[F-:red]\hbox{$H^3(K)$} \\
    &  \longrightarrow\ s
  }
\]
\caption{The $E_2$ page of the Serre spectral sequence associated to the fibre sequence $BG_0 \to BG\to BK$.
Boxed entries lie on the line $s+t=3$.}\label{fig:E2}
\end{figure}

Figure~\ref{fig:E2} shows part of the $E_2$ page, making the substitutions that $H^0(K,A) = V^K$, the invariants of the $K$-action on the module $A$, and that $H^1(BG_0)=H_1(BG_0)=0$, as $G_0$ is connected, giving the zeroes in the 1-row.
The asterisk entries are not needed.

Note that by the universal coefficient theorem there is a short exact sequence
\[
  0\to \Ext(H_2(K),\ZZ)\to H^3(K)\to \Hom(H_3(K),\ZZ)\to 0,
\]
hence if $H_3(K)$ is torsion, the right term vanishes, and so $H^3(K)$ is torsion, by the assumption on the Ext group.

The short exact sequence $H^2(BG_0)_\tors \into H^2(BG_0)\to H^2(BG_0)/\tors$ of groups is also a short exact sequence of $K$-modules, so there is a long exact sequence which reads in part
\[
  \cdots \to H^1(K,\cH^2(BG_0)_\tors) \to H^1(K,\cH^2(BG_0))\to H^1(K,\cH^2(BG_0)/\tors)\to\cdots
\]
As $\cH^2(BG_0)_\tors$ is finite, the left term is torsion, and so it suffices to know that $H^1(K,\cH^2(BG_0)/\tors)$ is torsion to conclude that $H^1(K,\cH^2(BG_0))$ a torsion group.

There is a version of the universal coefficient theorem (UCT) for cohomology with local coefficients, that generally only holds under slightly special extra assumptions.
Notice that while everything here is in terms of group cohomology of $K$, this is the same as the singular cohomology of $BK$
Finding this UCT was set as an exercise by Spanier \cite[Ch.\ 5, Exercise J.4]{Spanier_66}, and can be completed by using \cite[Ch VI, Theorem 3.3a]{CartanEilenberg} under the special assumption (3a') in \emph{loc.\ cit.}.\footnote{The formulation I use here was given by Oscar Randal-Williams at \url{https://math.stackexchange.com/q/4135432} (version: 2021-05-11).} 
The special assumptions are satisfied by the setup here, as we are working over the integers, and where the local coefficient system $\cH^2(BG_0)/\tors$ has fibres that are finitely generated free abelian groups.
The latter fact means we can also rely on the local system being reflexive, in that it is naturally isomorphic to its double dual.

The upshot is, there is a short exact sequence
\[
  0\to \Ext(H_0(K,A^*),\ZZ)\to H^1(K,A)\to \Hom(H_1(K,A^*),\ZZ)\to 0
\]
where $A^*=\Hom_\ZZ(A,\ZZ)$ is the dual module, also with underlying abelian group isomorphic to $\ZZ^n$.
Then $H_0(K,A^*)$ is a quotient of $A^*$, hence is finitely generated, and so $\Ext(H_0(K,A^*),\ZZ)$ is torsion (as it only sees the torsion subgroup of $H_0(K,A^*)$, which is finite).
Thus if $H_1(K,A^*)$ is torsion, the right term vanishes, and so $H^1(K,A)$ is torsion, as needed, where we take $A=\cH^2(BG_0)/\tors$.

We thus know that all the boxed groups are torsion groups, and so the groups $E_\infty^{s,t}$ for $s+t=3$, being subquotients of them, are also torsion.
Further, this implies that $H^3(BG,\ZZ)$, being an iterated extension of torsion group, is torsion, as needed.
\end{proof}

Note that nontrivial coefficients for $H_1$ are essential in this result, since the action of $K$ on $\cH^2(BG_0)$ is nontrivial in general. 

\begin{example}\label{ex:nontrivial_action_of_component_grp}
For the group $G=U(1)\rtimes \Aut((U(1))$, the $\Aut(U(1))\simeq \{\pm1\}$-action on $\cH^2(BU(1))\simeq \ZZ$ is nontrivial, as the inversion map on $U(1)$ induces a map on $BU(1)$ that pulls back the Chern class of the universal line bundle to its negative.
More generally, given a group $K$ with a surjection $K\to \{\pm1\}$, we get a nontrivial semidirect product $U(1)\rtimes K$ where $\cH^2(BU(1))$ is a nontrivial $K$-module.
\end{example}

It is not immediately obvious what the action can be for our class of examples appearing in Example~\ref{ex:loc_fin_group_bounded_torsion}, so this is not a redundant assumption.

\begin{rem}\label{rem:specific_torsion_info}
Given more specific information about the torsion in $H^*(BG_0,\ZZ)$, for instance if it vanishes, or is all $P$-torsion for some class $P$ of primes, and $K$ is a $P$-torsion group, then more information about what sort of torsion appears in $H^3(BG,\ZZ)$ is available.
\end{rem}

Recall that a \emph{locally finite group} is a discrete group that is the filtered colimit of finite groups, or, equivalently, the direct limit of its finite subgroups.
For instance, an infinite direct sum of finite groups is locally finite.

\begin{cor}\label{cor:loc_finite_group_torsion_Ext}
Let $K$ be a locally finite group such that $\Ext(H_2(K),\ZZ)$ is torsion.
If $G$ is a finite-dimensional Lie group with $\pi_0(G)\simeq K$, then $H^3(BG,\ZZ)$ is torsion.
\end{cor}

\begin{proof}
The key fact we need is that homology of groups commutes with taking filtered colimits, and this is true even with nontrivial coefficients, assuming the coefficients are also expressed as a direct limit.
For a locally finite group $K$ acting on an abelian group $A$, we can restrict the action on $A$ to finite subgroups $K_\alpha\subset K$, and so then $A$ as a $K$-module is the colimit of $K_\alpha$-modules.

Now since the group homology of a finite group is torsion, the homology groups $H_n(K,\ZZ)$, $n=1,3$ of a locally finite group are also torsion. 
By the observation about homology with nontrivial module coefficients, we can even conclude that $H_1(K,A)$ is torsion for $A$ a $K$-module.

Thus if we also know that $\Ext(H_2(K),\ZZ)$ is torsion, then we have all the hypotheses for Lemma~\ref{lemma:torsion_extensions}, and so we can conclude that $H^3(BG,\ZZ)$ as required.
\end{proof}

With Lemma~\ref{lemma:torsion_extensions} in hand, we can prove the sharper version of the main result.

\begin{thm}\label{thm:main_thm_sharp}
Fix a finite-dimensional Lie group $G$ and a connected topological space $X$ such that there is a finite-sheeted covering space $Y\to X$, with fundamental group satisfying the following condition: every homomorphism $\pi_1(Y)\to \pi_0(G)$ factors through a group $K$ satisfying the algebraic properties in Lemma~\ref{lemma:torsion_extensions}.
Then for any central extension $U(1)\to \widehat{G}\to G$, and principal $G$-bundle $P\to X$, the associated lifting gerbe is torsion.
\end{thm}

\begin{proof}
The proof is similar to that for Theorem~\ref{thm:main_thm_easy}, except instead of reducing the structure group to $G_0$, we proceed as follows.

Fix a central extension $U(1)\to \widehat{G}\to G$.
First note that the lifting bundle gerbe for the universal $G$-bundle $EG\to BG$ has $DD$-class given by the the class in $H^3(BG,\ZZ)$ classifying the central extension, by Lemma~\ref{lemma:DD-facts}.\ref{lemma:universal_lifting_gerbe} applied to the identity map of $BG$.

Then, given a group $K$ satisfying the hypothesis of Lemma~\ref{lemma:torsion_extensions}, and a homomorphism $K\to \pi_0(G)$, the pullback group $G_K := K\times_{\pi_0(G)} G$ has a central extension $U(1)\to K\times_{\pi_0}\widehat{G}\to G_K$. 
Since $H^3(G_K,\ZZ)$ is torsion (as $\pi_0(G_K) \simeq K$) all lifting gerbes associated to $G_K$-bundles are torsion.
Further, the lifting gerbe of the pulled-back $G$-bundle $BG_K\times_{BG} EG \to BG_K$ is also torsion.
Hence if the classifying map of a $G$-bundle factors through $BG_K$, even up to homotopy, then the $G$-bundle has torsion lifting gerbe.

Now fix a $G$-bundle $P\to X$, take the covering space $Y\to X$ as in the statement of the theorem, and the classifying map $Y\to X \to BG$.
Now there is a diagram commuting (up to homotopy)
\[
  \xymatrix{
    Y\ar[rr] \ar[d]&& BG \ar[d]\\
    B\pi_1(Y) \ar[rr]  \ar[dr] && B\pi_0(G)\\
    & BK \ar[ur]
  }
\]
where $K$ is as in the theorem statement.
Since $BG_K$ is the (homotopy) pullback $BK\times_{B\pi_0(G)}BG$, the classifying map $Y\to BG$ factors through $BG_K$, up to homotopy.
Thus the lifting gerbe for $Y\times_X P\to Y$ is the pullback of the torsion lifting gerbe on $BG_K$, and hence is torsion.

The rest of the proof of Theorem~\ref{thm:main_thm_easy} applies, to arrive at the conclusion that $DD(P)$ is torsion.
\end{proof}

\begin{rem}
The hypothesis on homomorphisms $\pi_1(Y) \to \pi_0(G)$ is stated so that one recovers both the cases $K=\pi_1(Y)$ and $K\subseteq \pi_0(G)$, which are what will presumably occur most often in practice.
\end{rem}

\section{Examples and applications}

\subsection{Examples}

Theorem~\ref{thm:main_thm_sharp} applies fairly trivially when $G$ is a finite-dimensional Lie group with $\pi_0(G)$ finite, so the interesting cases are when $\pi_0(G)$ is infinite.

\begin{cor}
If $X$ is a topological space with a locally finite fundamental group $\pi=\pi_1(X)$, with $\Ext\left(H_2(\pi),\ZZ\right)$ torsion, and $G$ is a finite-dimensional Lie group, then every lifting gerbe arising from a $G$-bundle on $X$ is torsion.

Similarly, if $G$ is a finite-dimensional Lie group with $\pi=\pi_0(G)$ locally finite with $\Ext\left(H_2(\pi),\ZZ\right)$ torsion, then for any topological space $X$, every lifting gerbe associated to a $G$-bundle on $X$ and central extension of $G$ is torsion.
\end{cor}

The main difficulty thus lies in arranging that the $\Ext$ group is torsion. 
But there are certainly infinite groups where this is so.

\begin{lemma}\label{lemma:loc_fin_group_bounded_torsion}
Let $K$ be a locally finite group given by a countable direct limit of finite abelian groups $K_\alpha$, the exponents of which are bounded, but whose orders are unbounded.
Then $K$ is an infinite group satisfying the hypotheses of Lemma~\ref{lemma:torsion_extensions}.
\end{lemma}

\begin{proof}
This implies $K$ has finite exponent, namely the least common multiple of the exponents of the groups $K_\alpha$, and also implies that $H_2(K,\ZZ) \simeq \colim_\alpha H_2(K_\alpha,\ZZ)$ is torsion, since each $H_2(K_\alpha,\ZZ)$ is finitely generated and torsion, hence finite. 
Further, the torsion in $H_2(K_\alpha,\ZZ)$ is bounded in terms of the exponent of $K_\alpha$, so there is a uniform bound on the torsion as $\alpha$ varies.

Using the fact that a countable directed diagram has a cofinal subsequence $(A_n)_{n\in\NN}$, the colimit defining $K$ can be assumed without loss of generality to be sequential.
We can then consider the exact sequence (all homology with integer coefficients)
\[
  0\to \mathop{\lim\nolimits^1}\Hom(H_2(K_n),\ZZ)\to \Ext(H_2(K),\ZZ) \to \lim_{\NN}\Ext(H_2(K_n),\ZZ)\to 0.
\]
Now $\Hom(H_2(K_n),\ZZ)=0$ for all $n$, as $H_2(K_n,\ZZ)$ is finite, hence the $\mathop{\lim\nolimits^1}$ term is trivial.
Now we use the fact that the sequence $H_2(K_n,\ZZ)$ has bounded exponent, hence $\Ext(H_2(K_n),\ZZ)$ has bounded torsion, so that $\lim_{\NN}\Ext(H_2(K_n),\ZZ)$ is still torsion.
This, finally, gives us that $\Ext(H_2(K,\ZZ),\ZZ)$ is torsion, and hence we can apply Corollary~\ref{cor:loc_finite_group_torsion_Ext} to conclude.
\end{proof}

\begin{example}\label{ex:loc_fin_group_bounded_torsion}
For instance, consider the infinite direct sum $\bigoplus_{n\in\NN} A$ for any finite abelian group $A$.
This is locally finite, as it is the colimit of the sugroups given by finite direct sums, and its exponent is same as that of $A$.

More generally, if $K_0 \subset K_1 \subset K_2\subset \cdots$ is a countable increasing sequence of finite abelian groups with eventually constant exponent $e$, then $\bigcup_n K_n$ is locally finite with exponent $e$.
\end{example}

In the other direction, here is a counterexample to Theorem~\ref{thm:main_thm_sharp} that was used in \cite{Murray_Stevenson_11}. The specific hypothesis of \cite[Theorem 3.6]{Murray_Stevenson_11} that is violated---connectivity of $G$---is not one required for the version of the theorem here, but it still breaks our hypotheses in other ways.

\begin{example}[{\cite[\S4.1]{Bryl} and \cite[\S3.5]{Stu_phd}}]\label{ex:cup_product}
Take a principal $U(1)$-bundle $Q\to X$ and a function $X\to U(1)$, with classes $\alpha\in H^2(X,\ZZ)$ and $\beta\in H^1(X,\ZZ)$.
The bundle gerbe classified by the cup product $\alpha\cup\beta$ is given by the lifting bundle gerbe corresponding to the central extension
\[
  U(1) \to (U(1)\times \ZZ) \tilde\times U(1) \to U(1)\times \ZZ
\]
with product $(n_1,z_1;w_1)\cdot (n_2,z_2;w_2) = (n_1+n_2,z_1z_2;w_1w_2z_1^{n_2})$ and the principal $(U(1)\times \ZZ)$-bundle $Q\times_X f^*\RR\to X$.\footnote{Treating $\RR\to U(1)$ as a principal $\ZZ$-bundle. Note that there is a typo the definition of the product in \cite{Murray_Stevenson_11}, we use the definition in \cite[Equation (4-10)]{Bryl}}

Then if $X=M\times S^1$, then every every finite-index subgroup $H < \pi_1(M\times S^1)\simeq \pi_1(M)\times \ZZ$ admits a surjection onto a finite-index subgroup of $\ZZ$, for example the one induced by the composite $Y\to M\times S^1\xrightarrow{\pr_2} S^1$
Thus for every finite-sheeted covering space $Y\to M\times S^1$, every nontrivial homomorphism $\phi\colon \pi_1(Y)\to \pi_0(U(1)\times \ZZ)=\ZZ$ fails to factor through a subgroup satisfying Lemma~\ref{lemma:torsion_extensions}, as $\im(\phi)\simeq \ZZ$ and $H_1(\ZZ,\ZZ)=\ZZ$.

If $H^2(M,\ZZ)$ contains non-torsion classes, then by K\"unneth we have $H^3(M\times S^1,\ZZ) \simeq H^2(M,\ZZ)\otimes H^1(S^1,\ZZ)\oplus H^3(M,\ZZ)\otimes H^0(S^1,\ZZ)$, which means there are cup product classes that are non-torsion and corresponding to finite-dimensional lifting gerbes by the above construction.
\end{example}

However, what happens if we are in the situation where we know $K$ is locally finite, and so Corollary~\ref{cor:loc_finite_group_torsion_Ext} has a chance of applying?

\begin{example}\label{ex:pullback_extension_from_K}
Even if $K$ is locally finite, if there is a non-torsion class in $H^3(K,\ZZ)$, this gives rise to a nontrivial central extension  $U(1) \to \widehat{K}\to K$.
Given any semidirect product $G = G_0\rtimes K$, we can form the pullback $G\times_K \widehat{K}$, which is then a nontrivial central extension of $G$, and in fact is classified by a non-torsion class in $H^3(BG,\ZZ)$, using the fact $H^3(K,\ZZ)\to H^3(BG,\ZZ)$ has a retraction induced by the section of $G\to K$.
\end{example}

How might one get non-trivial classes in $H^3(K,\ZZ)$ for a locally finite group $K$? If there is some nontrivial $c\in \Ext(H_2(K),\ZZ)$, then using the exact sequence
\[
  0=\Hom(H_2(K),\RR) \to \Hom(H_2(K),\RR/\ZZ) \xrightarrow{\simeq} \Ext(H_2(K),\ZZ) \to \Ext(H_2(K),\RR) = 0
\]
we get a nontrivial homomorphism $\phi_c\colon H_2(K)\to \RR/\ZZ\simeq U(1)$ (here using the fact $\RR$ is divisible, hence an injective group, and $H_2(K)$ is torsion).
Further, since $H^2(K,H_2(K)) \to \Hom(H_2(K),H_2(K))$ has a section, by the universal coefficient theorem for group homology, the identity map on $H_2(K)$ coresponds to a nontrivial central extension $H_2(K) \to \widehat{K} \to K$ classified by  class in $H^2(K,H_2(K))$.
We can then induce a central extension of $K$ by $U(1)$ using the homomorphism $\phi_c$, and central extensions by $U(1)$ are classified by $H^3(K,\ZZ)$.
All of these constructions are induced by homomorphisms, so all up we have an injective homomorphism 
\[
\Ext(H_2(K),\ZZ)\simeq \Hom(H_2(K),U(1)) \into H^2(K,U(1)) \xrightarrow{\simeq} H^3(K,\ZZ) \to H^3(BG,\ZZ),
\]
and so non-torsion elements in $\Ext(H_2(K),\ZZ)$ are a source of non-torsion elements in $H^3(BG,\ZZ)$.

Now the $E_3$ page of the spectral sequence in Figure~\ref{fig:E2} has a potentially nontrivial differential $d_3^{0,2}\colon \cH^2(BG_0)^K\to H^3(K,\ZZ)$.
Moreover $\im(d_3^{0,2})$ is the kernel of the edge homomorphism $H^3(K,\ZZ) \to H^3(BG,\ZZ)$, since $H^3(K,\ZZ)/\im(d_3^{0,2}) \simeq E_\infty^{3,0} \subset H^3(BG,\ZZ)$.
Without more specific analysis of the differential $d_3^{0,2}$, though, it is not immediately obvious that a non-torsion element remains non-torsion in $H^3(BG,\ZZ)$, in general.

However, some simple examples can be calculated directly, for instance:

\begin{example}
Take $K$ some locally finite group with a surjection $\alpha\colon K\to \{\pm1\}$, and define $G = U(1)\rtimes K$ as in Example~\ref{ex:nontrivial_action_of_component_grp}, so that $\cH^2(BU(1))^K=\{0\}$, and so $\Ext(H_2(K),\ZZ)$ injects into $H^3(BG)$.
\end{example}

We can even give a concrete example of a locally finite group $K$ where $\Ext(H_2(K),\ZZ)$ has lots of non-torsion elements, giving non-torsion elements in $H^3(U(1)\rtimes K)$.
This is despite $\pi_0(U(1)\rtimes K)$ being purely torsion, in contrast to $U(1)\times \ZZ$ in Example~\ref{ex:cup_product}.

\begin{example}
Define our locally finite group to be 
\[
  K := \bigoplus_{p\text{ prime}} (\ZZ/p)^2
\]
which is the colimit of the groups $K_p:=(\ZZ/2)^2\oplus (\ZZ/3)^2\oplus \cdots \oplus (\ZZ/p)^2$.
It is a classical result that $H_2((\ZZ/p)^2,\ZZ)\simeq \ZZ/p$, and we get\footnote{This and the `classical result' follow from the formula $H_2(G\times H,\ZZ)\simeq H_2(G,\ZZ)\oplus H_2(H,\ZZ)\oplus (G^{ab}\otimes H^{ab})$ and the fact cyclic groups have vanishing $H_2$.}
\[
  H_2(K,\ZZ) \simeq \bigoplus_{p\text{ prime}} \ZZ/p.
\]
Since $\Ext$ sends direct sums to direct products, we have that
\[
  \Ext(H_2(K,\ZZ),\ZZ) \simeq  \prod_{p\text{ prime}}\ZZ/p,
\]
and so $H^3(K,\ZZ)$ and hence $H^3(U(1)\rtimes K,\ZZ)$ has plenty of non-torsion elements.

For the sake of concreteness, one can take the element $(1,0;1,0;1,0;\ldots)\in \Ext(H_2(K),\ZZ)$ and from this induce a central extension
\[
  U(1) \to \mathcal{T} \to U(1)\rtimes K
\]
with non-torsion characteristic class in $H^3(B(U(1)\rtimes K),\ZZ)$, and then this class is the $DD$-class of a lifting bundle gerbe following Lemma~\ref{lemma:DD-facts}.\ref{lemma:universal_lifting_gerbe}.
\end{example}

As a result of this example, we can see that the assumption of torsion $\Ext(H_2(K),\ZZ)$ in Corollary~\ref{cor:loc_finite_group_torsion_Ext} is necessary if one wants to make a statement about all possible finite-dimensional Lie groups with group of connected components $K$.

\subsection{Applications}

These results give strong obstructions to finding finite-dimensional lifting bundle gerbes.
Namely, take a manifold or space  $X$ satisfying the hypotheses of Theorem~\ref{thm:main_thm_sharp} with $H^3(X,\ZZ)$ torsion-free. 
Then all finite-dimensional lifting gerbes on $X$ have a stable trivialisation.

An example of particular focus in the literature considers a generic connected, compact simple Lie group, not necessarily simply-connected, which always has nontrivial third integral cohomology.

\begin{prop}\label{prop:simple_Lie_grp_lift_gerbe}
The only connected, compact simple Lie groups $G$ that admit a nontrivial lifting bundle gerbe are the groups $PSO(4n)$ for $n>1$, and in this case there is precisely one, up to stable isomorphism\footnote{It is enough to know that stable isomorphism is equivalent to having the same $DD$-class}.
\end{prop}

\begin{proof}
First recall that for a Lie group as in the proposition, the fundamental group is finite, so we can apply Theorem~\ref{thm:main_thm_easy}.
Then for all $G\neq PSO(4n)$, $H^3(G,\ZZ)\simeq \ZZ$, hence any lifting gerbe must be trivial, and for $PSO(4)\simeq SO(3)\times SO(3)$, $H^3\simeq \ZZ\oplus \ZZ$, with the same conclusion.

For $H=PSO(4n)$ with $n>1$, as $H^3(PSO(4n),\ZZ)\simeq \ZZ\oplus \ZZ/2$, there is only one nontrivial torsion class, and this \emph{can} be can realised by a lifting gerbe, following \cite[Remark 5.1]{Krepski}.
Namely, there is a nontrivial central extension
\[
  U(1) \to \widehat{Z}\to \ZZ/2\times \ZZ/2 =: Z
\]
where the underlying manifold of $\widehat{Z}$ is $U(1)\times \ZZ/2\times \ZZ/2$, with multiplication arising from the nontrivial 2-cocycle $\beta \colon Z\times Z\to U(1)$ with $\beta((n_1,n_2),(m_1,m_2)) = (-1)^{n_1+m_2}$ \cite{Gawedzki-Waldorf}.
Then there is a nontrivial lifting bundle gerbe arising to the $(\ZZ/2\times \ZZ/2)$-bundle $\Spin(4n)\to PSO(4n)$, and this has torsion $DD$-class in $H^3(PSO(4n),\ZZ)$.
\end{proof}

The following corollary is immediate, as a basic gerbe on a Lie group is, by definition, a non-torsion generator of $H^3(G,\ZZ)$ \cite{Gawedzki-Reis}.

\begin{cor}
No basic gerbe on a connected, compact simple Lie group $G$, not necessarily simply-connected, can be constructed as a lifting bundle gerbe.
\end{cor}

\begin{rem}
Note however, that the pullback of a basic gerbe can be a finite-dimensional lifting bundle gerbe.
For example, the basic gerbe on $SU(2)$ can be constructed via a carefully constructed submersion, and is not a lifting gerbe.
Its pullback along the Weyl map $U(1)\times SU(2)/U(1) \to SU(2)$ (where $U(1)\subset SU(2)$ is the diagonal matrices), however, is stably isomorphic to the Weyl bundle gerbe \cite{Becker-Murray-Stevenson}, which \emph{is} a lifting bundle gerbe.
\end{rem}

This theorem as algebraic consequences, as well.
Recall that a \emph{crossed module} of Lie groups consists of a homomorphism $t\colon \hat{K} \to L$ of Lie groups such that $\ker(t) \to \hat{K} \to t(\hat{K}) =: K$ is a central extension, $K$ is a normal closed subgroup of $L$, and a lift of the adjoint action of $L$ on $K$ to $\hat{K}$, such that $\hat{K}\to L\to \Aut(\hat{K})$ agrees with the adjoint action of $\hat{K}$ on itself.
Moreover, $L\to \coker(t)$ is a principal $K$-bundle, and we can consider the corresponding lifting bundle gerbe.

\begin{cor}\label{cor:no-go_for_mult_gerbes}
Let $G$ be a finite-dimensional connected, compact simple Lie group.
Then there is no finite-dimensional crossed module $t\colon \hat{K}\to L$ of Lie groups with $\ker(t)\simeq U(1)$ and $\coker(t)\simeq G$ whose associated lifting bundle gerbe is nontrivial.
More generally, no multiplicative bundle gerbe on $G$ can be stably isomorphic to a finite-dimensional lifting bundle gerbe with non-torsion $DD$-class.
\end{cor}

\begin{proof}
By Proposition~\ref{prop:simple_Lie_grp_lift_gerbe} if $\hat{K}\to L$ is a crossed module with $\coker(t)=G$, then the corresponding lifting bundle gerbe must be torsion, hence we can consider just the case of $G=PSO(4n)$, $n>1$.

The lifting bundle gerbe arising from a crossed module $\hat{K}\to L$ is \emph{multiplicative} \cite{CJMSW}, as it fits into a strict 2-group extension corresponding to the extension of crossed modules
\begin{equation*}\label{eq:xmod_ext}
  \xymatrix@R=3ex@C=9ex{
    U(1) \ar[d] \ar[r] & \hat{K} \ar[d] \ar[r] & 1 \ar[d]\\
    1 \ar[r] & L \ar[r] & \coker(t)
  }
\end{equation*}
of $\coker(t)$ by $\mathbf{B}U(1)$ (as strict Lie 2-groups).
By \cite[Proposition 5.2]{CJMSW}, the $DD$-class of a multiplicative bundle gerbe on $G$ is in the image of the transgression map $H^4(BG,\mathbb{Z})\to H^3(G,\mathbb{Z})$. 
But by \cite[Theorem 6]{Henriques_17}, for any connected compact Lie group $G$, the restriction map $H^4(BG,\mathbb{Z}) \to H^4(BT,\mathbb{Z})$ induced by a maximal torus $T\subset G$ is injective, hence $H^4(BG,\mathbb{Z})$ is torsion-free, including for $G=PSO(4n)$. 
Thus the only class in $H^3(PSO(4n),\mathbb{Z})$ represented by some lifting gerbe is not the $DD$-class of the lifting gerbe arising from a crossed module.

The more general statement follows as multiplicative bundle gerbes on $G$ are classified by $H^4(BG,\ZZ)$ \cite[Proposition 5.2]{CJMSW} (see also \cite{Gawedzki-Waldorf}, where the image of the injective map $H^4(BG,\ZZ)\to H^4(B\tilde{G},\ZZ)$ is characterised for all possible compact connected simple Lie groups).
\end{proof}

This tells us that any finite-dimensional crossed module $t\colon \hat{K}\to L$ as in Corollary~\ref{cor:no-go_for_mult_gerbes}, there is a principal $U(1)$-bundle $T\to L$ whose restriction to $K \subset L$ is $\hat{K}\to K$, and there is a $\hat{K}$-action of $T$ covering the action on $L$ by multiplication via $t$.
Note that this corollary places a constraint on the structure of a crossed module analogous to how one can have group extensions that are topologically trivial, but still algebraically nontrivial---there may still be algebraically nontrivial crossed modules with cokernel $G$.

As Corollary~\ref{cor:no-go_for_mult_gerbes} constrains the structure of finite-dimensional crossed modules, it means that higher geometry of principal 2-bundles with strict structure 2-group (i.e.\ a crossed module), as in \cite{Waldorf_18} for example, really must use infinite-dimensional constructions.
A connection for such a bundle takes its values in the truncated $L_\infty$-algebra that is the crossed module of Lie algebras associated to the given crossed module of Lie groups \cite[Definition 5.1.1]{Waldorf_18}.
As a result of the preceeding corollary, finite-dimensional crossed modules $\tau\colon \hat{\mathfrak{k}} \to \mathfrak{l}$ of Lie algebras where $\coker(\tau)$ is simple are insufficient to capture all examples of interest.

Finally, recall that twisted $K$-theory is a particular cohomology theory generalising topological $K$-theory, where one has an extra piece of data, a \emph{twist}; homotopy-theoretically this is a map to the classifying space for bundles of spectra with fibre the $K$-theory spectrum.
The most well-studied twists\footnote{There are other twists, which arise from bundles of certain self-absorbing $C^*$-algebras} arise from a factor of $K(\ZZ,3)$ of this classifying space, hence when constructing $K$-theory using geometric objects, geometric objects classified by maps to $K(\ZZ,3)$ are used. 
As such, bundle gerbes are one model for twists, as are principal $PU(\cH)$-bundles.
The latter give rise to lifting bundle gerbes, and as noted in the introduction, every bundle gerbe is (non-canonically) the lifting gerbe of some $PU(\cH)$-bundle.
But finite-dimensional bundle gerbes coming from $PU(n)$-bundles can be used as well, and more generally, the lifting bundle gerbes of principal $G$-bundles for other groups $G$.

\begin{cor}
Given a space $X$ as in Theorem~\ref{thm:main_thm_sharp}, a finite-dimensional $G$-bundle $P$ and central extension $\widehat{G}$, every twist $\tau=\tau(P,\widehat{G})$ of the $K$-theory of $X$ coming from this data is a torsion twist.
In particular if $H^3(X,\ZZ)$ is torsion-free, the resulting twisted $K$-theory $K^{\ast,\tau}(X)$ is just ordinary $K$-theory.
\end{cor}

We can view this result as putting strong constraints on the maps
\[
  H^1(X,G)\times H^3(BG,\ZZ) \to \{\text{twists of $K$-theory on $X$}\}.
\]
For instance in the case that $\pi_1(X)$ is locally finite and $\Ext(H^2(\pi_1(X)),\ZZ)$ is torsion, bundles for arbitrary finite-dimensional Lie groups give \emph{only} torsion twists.

\printbibliography

\end{document}